\documentclass[]{interact}
\usepackage{comment}
\usepackage{epstopdf}
\usepackage[caption=false]{subfig}

\usepackage{autonum}
\usepackage{xcolor}
\usepackage{placeins}
\usepackage{xspace}
\usepackage{tikz}
\usepackage{todonotes}
\usepackage{longtable}
\usepackage{tabularx}
\usetikzlibrary{arrows,positioning,shapes.misc,shapes.geometric}

\newcommand{\juan}[1]{\todo{J}\footnote{Juan: #1}\xspace}

\newcommand{\C}{\operatorname{\mathcal C}}
\usepackage{algorithm2e}
\usepackage{float}
\usepackage{apacite}

\theoremstyle{plain}
\newtheorem{theorem}{Theorem}[section]
\newtheorem{lemma}[theorem]{Lemma}

\theoremstyle{definition}

\theoremstyle{remark}

\begin{document}
\articletype{}
\title{Technical Note: Using Column Generation to Solve Extensions to the Markowitz Model}

\author{
\name{Lorenz M. Roebers, Aras Selvi\thanks{CONTACT Aras Selvi. Email: a.selvi@tilburguniversity.edu}, and Juan C. Vera}
\affil{Department of Econometrics and Operations Research, Tilburg University, 5037 AB Tilburg, The Netherlands}
}
\maketitle

\begin{abstract}
We introduce a solution scheme for portfolio optimization problems with cardinality constraints. Typical portfolio optimization problems are extensions of the classical Markowitz mean-variance portfolio optimization model. We solve such type of problems using a method similar to column generation. In this scheme, the original problem is restricted to a subset of the assets resulting in a master convex quadratic problem. Then the dual information of the master problem is used in a sub-problem to propose more assets to consider. We also consider other extensions to the Markowitz model to diversify the portfolio selection within the given intervals for active weights.
\end{abstract}


\section{Introduction}
In portfolio optimization, an investor allocates funds among available assets. The objective is to select the best portfolio among a set of feasible portfolios, where the quality of a portfolio is measured in terms of different factors, classically expected return and risk. That is, the portfolio optimization problem is, naturally, a multi-objective problem where there is a trade-off between risk and return: typically a higher expected return implies facing higher risk and vice versa. This trade-off is set according to the investor's risk aversion. A standard model for portfolio optimization is the traditional Markowitz Mean-Variance Portfolio Problem \cite{markowitz}. In this model the trade-off between expected return and risk is represented by a weighted combination of return expectation and return variance. The Markowitz model is theoretically very strong, but has received a lot of criticism since the setting is not realistic, thus it is important to extend the simple Markowitz model with \emph{cardinality} and \emph{quantity constraints}  (see, e.g., \cite{cesarone2013new}).

In this paper we consider the traditional Markowitz Mean-Variance Portfolio Problem extended by some practical constraints including \emph{cardinality} and \emph{quantity constraints}. Existing methods for tackling such hard constraints are based on heuristics and evolutionary computing. 
We propose a novel methodology based on a column generation approach to quadratic optimization problems. To do this the constraints are divided into two groups. The first group, the `easy' ones, are a set of linear constraints. The second group of constraints are  the hard ones in terms of computational complexity. This group consist of \emph{tracking error constraints}, which are not convex, and the \emph{cardinality constraints}, which are of combinatorial nature.


The basic model includes constraints setting limits on the assets' weights based on multiple features such as active weights, market capital quantile, sector and deviation from a benchmark. We define the \textit{basic model} to be the Markowitz model extended by the first group of constraints. As all the constraints included are linear, the basic model is a (Convex) Quadratic Optimization Problem, and can be solved very efficiently to optimality. Indeed, our experimental analysis indicates that commercial solvers are easily able to solve this problem. Our methodology will be based on solving iteratively several instances of the basic model, and thus efficiency is critical. 

 
The second group consists of three type of constraints,
\emph{active share constraints} on the deviation from the benchmark, \emph{tracking error constraints} on the correlation between the returns of the selected portfolio and the benchmark, and \emph{cardinality constraints} on the number of active assets in the portfolio.
We tackle the \textit{tracking error constraints} by continuous adjustment of the risk-parameter.
To handle the \emph{cardinality constraints}, we propose a novel asset selection sub-problem based on the marginal effect of investing in those assets. Our method could be considered an extension of column generation to the quadratic setting.

In the literature significant attention is paid to real-life trading costs and monitoring availability; in particular, cardinality constraints are studied. \citeA{cesarone2013new} provide a discussion on the computational complexity of this type of problems. While the classical Markowitz model is a convex quadratic programming model, the cardinality constraint results in a significantly more complex NP-hard problem, which can be modeled via a mixed-integer quadratic program. An exact approach is provided by \citeA{bienstock1996computational}, with a branch-and-cut algorithm. Even though this work provides theoretically strong results, such an approach is not practical for real-life problems (see \cite{cesarone2013new}). Therefore, the algorithms considered in the literature are mainly based on local search and multi-objective evolutionary algorithms, which fail to guarantee global optimality. The effect of genetic algorithms, tabu search and simulated annealing on the cardinality constraint is seen in the detailed work of \shortcite{chang00} and references therein. Extensions to evolutionary algorithms, such as memetic algorithms are also studied for tackling the cardinality constraint \cite{streichert04}. More experiments and comparisons in evolutionary algorithms can be found in the work of \citeA{anagnostopoulos11}. 

The rest of this paper is structured as follows. In Section \ref{prelim} we provide the notation and formulate the problem mathematically. In Section \ref{sec:soltn} we present the methodology to solve the problem. In Section \ref{appl} we perform a numerical test of our methods. We describe the data we use and evaluate the performance of our solution by analyzing the results. The conclusion, alternative methods to adjust the solution and final remarks are in Section \ref{disc}.

\section{Preliminaries} \label{prelim}
The portfolio construction problem follows the model of \citeA{markowitz}, where a risk averse investor's goal
is to construct a portfolio maximizing expected return and minimizing risk. Risk aversion is quite a realistic assumption given large experimental evidence involving, for instance, lotteries \cite{holt02}.
The Markowitz model uses the volatility of the portfolio returns as measure of risk. Given $\Omega$, the variance-covariance matrix of the assets' return, and $\alpha$, the vector of assets' expected returns, we formulate the following optimization model,
\begin{equation} \label{model:Markovitz}
\begin{aligned}
&& \min_{w} & \ w^T \Omega w - \lambda \alpha^T w \\
&& \text{s.t.} & \  w_i \geq 0 \quad \forall i \quad (\textit{Non-negative weight allocation})\\
&& & \ \sum_{i \in \mathcal{I}} w_i = 1 \quad (\textit{Full portfolio invested})
\end{aligned}
\end{equation}
where the vector of decision variables $w$ represents the percentage of wealth invested in each asset, and $\mathcal{I}$ is the set of all possible assets.
The parameter $\lambda >0$ reflects the investor risk aversion, balancing the preference between risk and return.
The first constraint restricts the weight allocation to be non-negative (short positions allowed), and the second constraint ensures that the total allocation of asset weights sums up to 1 (simply indicating that all the funds have to be invested in our asset universe).

It has been assumed that the investor cares here only about the meand and variance of portfolio returns and we neglect all the higher moments like skewness (tail risk) or kurtosis (fatness of tails). Thus, up until this point we have used the same assumptions under which in a frictionless environment a mean-variance efficient frontier can be easily constructed (see e.g. \cite{cochrane09}) and the optimal portfolio chosen.


\subsection{The Basic Model}\label{sec:basic}
We add constraints to model (\ref{model:Markovitz}) to bring the analysis closer to practical implementation and more recent developments in the literature. We use $MCAPQ_k$ to denote the set of assets corresponding to companies in the $k$-th market capital quintile, where $k \in \mathcal{K} = \{ 1,\ldots,5\}$. The index $k=1$ stands for the largest quintile while $k=5$ stands for the smallest. The assets in the given problem are being distributed into sectors: each asset $i$ belongs to some sector $j$. We define the set $sector_j$ as the set of all assets $i$ that are in sector $j$, where the relevant companies belong to sector $j \in \mathcal{J}$.

All of the extensions are financially quite intuitive. We assume that a benchmark
 with weights $w^b$ is exogenously specified. We introduce the auxiliary decision variable
$d=w-w^b$, which measures the deviation from the given benchmark in terms of assets' weights.

The first difference with model \eqref{model:Markovitz} is that we use $d^T \Omega d + \lambda \alpha^T d$ as the objective function instead of $w^T \Omega w + \lambda \alpha^T w$. In our approach, instead of considering $\lambda$ as an exogenously fixed parameter, the value of $\lambda$ is dynamically adjusted in order to satisfy risk exposure constraints. More details are provided in section 3.

We add constraints on the deviation from the benchmark according to different features:
\begin{equation}
\begin{array}{rrclcl}
-0.05  \leq&&  d_i  &\leq  0.05 &  \forall i \in \mathcal{I} &(\textit{Deviation from Benchmark Weight})\\
-0.1 \leq && \displaystyle \sum_{i \in sector_j} d_i & \leq  0.1 & \forall j \in \mathcal{J} &  (\textit{Sector Active Weight})\\
-0.1 \leq&& \displaystyle \sum_{i \in MCAPQ_k} d_i & \leq  0.1 &  \forall k \in \mathcal{K} &(\textit{Market Capital Quintile Active Weight})\\
-0.1 \leq&& 	 \displaystyle \sum_{i \in \mathcal{I}} d_i\beta_i & \leq  0.1 &&(\textit{Beta Active Weight})
\end{array}
\end{equation}
The \textit{deviation from benchmark weight constraint} restricts the individual deviation from the benchmark weight from -5\% to +5\%.  The assets in the given problem are being distributed into sectors; each asset belongs to a sector. Thus, the \emph{sector active weight constraint} restricts the total summed deviation for each of the sectors to be less than 10\%. 

The market capitalization of an asset captures the asset’s capitalization size relative to the market.The \textit{market cap quintile constraint} ensures that the total summed deviation per quintile capitalization size does not exceed 10\%. The \textit{beta active weight constraint} ensures that the total sum of the product of beta, an exogenously specified measure of each of the asset's sensitivity to the whole market, and the deviation from the benchmark, is restricted to no more than 10\%. These constraints ensure that the constructed portfolio does not deviate much from the benchmark, as well as that the weight allocation of the assets are distributed across sectors, capitalization and betas. This is relevant from a risk management perspective and other policies which limit the portfolio exposure to idiosyncratic (e.g. firm specific) shocks.
\subsection{Computationally hard constraints}\label{hardcons}
The basic model  discussed in section \ref{sec:basic} could be solved efficiently. In this section, the constraints we introduce are the ones which make the problem computationally intractable.
\subsubsection{Active share}
The first non-convex constraint we introduce is the \emph{active share constraint}. The ‘active share’ concept originally proposed by \citeA{cremers09} and further analyzed by \citeA{cremers17} is a measure of the relative activeness of a portfolio. There are several reasons why an investor could be interested in the active share of a portfolio.  For example, in mutual funds, it is very important to know how active the fund manager is. After all, any active portfolio management services involve certain costs (management fees) which are considered to be a compensation for a portfolio manager's effort to generate positive abnormal returns (in finance jargon, positive ‘alphas’). This example with a mutual fund manager might sound restrictive, but the concept is quite broad and could be applied in any portfolio selection procedure.

 Since by construction the portfolio fully invested in a benchmark has a zero active share, our optimization problem simply tries to find an optimal deviation from a given benchmark. Moreover, a deviation from the benchmark can be observed by comparing the amounts invested in each asset. The \textit{active share constraint} is then given by:
\begin{equation}\label{c8}
 0.6 \leq 1 - \displaystyle \sum_{i\in \mathcal{I}} min(w_i,w^b_i)	\leq  1 \qquad (\textit{Active share})
\end{equation}
    The mathematical correctness of this constraint is also intuitive. Assume the portfolio selected in the current solution is exactly the same as the benchmark. Then, $\min(w_i,w_i^b) = w_i^b$ for every asset $i$, resulting in $1 - \sum_{i} \min(w_i,w_i^b) = 0$. This is clearly a violation. However, if there are more significant deviations in the selected portfolio, then the constraint will be satisfied. We reformulate the \emph{active share constraint} in terms of the the sum of the absolute value of the deviations using lemma \ref{lemma}.
\begin{lemma}\label{lemma}
Under the Full Portfolio constraint, for any $a$:
\begin{equation}
\sum_i \min(w_i,w^b_i)\leq a \text{ if and only if }  \sum_i{|d_i|}\geq 2(1-a)
\end{equation}
\end{lemma}
\begin{proof}
For all $i$ we have $\min(w_i,w^b_i) =\tfrac 12 (w_i+ w^b_i) - \tfrac 12 |w_i - w^b_i| = \tfrac 12 (w_i+ w^b_i) - \tfrac 12 |d_i|$. Thus $\sum_i \min(w_i,w^b_i) = \tfrac 12 \sum_i(w_i+ w^b_i) - \tfrac 12 \sum_i|d_i| = 1 - \tfrac 12 \sum_i|d_i|$ and the lemma follows.
\end{proof}
Thus the \emph{active share constraint} is equivalent in our case to 
\begin{equation}
\begin{array}{rrclcl}
&\sum_i|d_i| & \geq & 1.2 && (\textit{Total absolute deviation})\\
\end{array}
\end{equation}
\subsubsection{Tracking error}
The tracking error calculation (introduced by \citeA{roll92}, analyzed by \citeA{rudolf99}) can also be seen as a measure of how the portfolio returns are dispersed relative to the benchmark. 
\begin{equation}
\begin{array}{rrclcl}
& 0.05 \leq \sqrt{d^T\Omega d}	& \leq & 0.1 && (\textit{Tracking error constraint})\\
\end{array}
\end{equation}

The left part of the constraint is concave and the right part is convex. Since it is typically difficult to solve a minimization problem with concave constraints, we use an alternative method to satisfy the \textit{tracking error constraint}, namely by adjusting the risk-aversion parameter values $\lambda$. Notice that the tracking error constraint is a constraint on the risk measure, and thus it makes sense that the selection of $\lambda$ will reflect the tracking error. Therefore the algorithm dynamically adjusts the value of $\lambda$ to satisfy the \textit{tracking error constraint}. The details for this approach are given in section \ref{sec:soltn}. 

\subsubsection{Cardinality}
The \emph{cardinality constraint} closely relates to the idea that financial markets are not frictionless and there are substantial transaction costs and divisibility limitations. The re-balancing of a portfolio which contains hundreds of assets can be extremely costly and erode all the net returns. This motivates to allow just a limited number of positions in the portfolio to change. Another idea is that a portfolio with fewer assets is easy to oversee and analyze. Thus, the cardinality constraint brings practical advantages. In our model, the cardinality constraint ensures that the number of active assets (i.e., assets with non-zero weight) is at least 50 and at most 70. 
\begin{equation}
\begin{array}{rrclcl}
&50  \leq  card(w_i \neq 0) & \leq & 70 && (Cardinality)\\
\end{array}
\end{equation}
This constraint is a combinatorial constraint, and adding this type of combinatorial constraint to a quadratic optimization problem makes the problem computationally very hard to solve. Our main contribution is a new methodology to tackle this constraint. In a nutshell, in the proposed methodology, interactively we maintain a small set of candidate assets to ensure the cardinality constraint. We find the optimal portfolio restricted to this set of assets. Then the set of assets is dynamically updated by computing the marginal effect each asset has on the objective, and including as new candidates the most promising assets, while dropping those of smaller weight. The marginal effect plays a role similar to the reduced cost in column generation for linear programs. Details are discussed in section \ref{sec:soltn}.

\section{Solution Approach}\label{sec:soltn}
In our methodology we divide the problem into a master problem and a sub-problem. Given a set $\C$ of candidate assets, we call the master problem the basic model introduced in section \ref{sec:basic} restricted to $\C$.
\begin{align}\tag{Psim$_{\C}$}\label{model:Psim}
    \begin{array}{rllll}
        \rho^{\C} = &\underset{d,w}{\min}\ & d^T \Omega d - \lambda d^T \alpha & &\\[0.5cm]
&\textrm{s.t.} & w_i &\geq 0  & \forall i \in \mathcal{I} \label{ct:1} \\ [0.3 cm]
& & \underset{i \in \mathcal{I}}{\sum} w_i &=  1 &\label{ct:2}\\[0.4 cm]
& & \left|d_i\right|  &\leq  0.05  & \forall i \in \mathcal{I} \label{ct:3}\\[0.25 cm]
& &\left|\underset{{i \in sector_j}}{\sum}  d_i\right|  &\leq 0.1 & \forall j \in \mathcal{J} \label{ct:4}\\[0.5 cm]
& & \left|\underset{{i \in MCAPQ_k}}{\sum} d_i \right|  &\leq  0.1  & \forall k \in \mathcal{K} \label{ct:5}\\[0.6 cm]
&  & \left|\underset{i \in \mathcal{I}}{\sum} d_i\beta_i \right|  &\leq  0.1 & \label{ct:6} \\[0.35 cm]
    \end{array}
\end{align}
The master problem \ref{model:Psim} is a (convex) quadratic optimization problem. There are efficient ways to solve this type of problems; e.g., interior point methods \cite{wright1997primal}. 

Now we need to take care of the hard constraints introduced in section \ref{hardcons}. First we look at the \emph{active share constraint} or equivalently the \emph{total absolute deviation constraint}. To fulfill this constraint, we derive the following sufficient condition. 
\begin{equation}
    \underset{i\in I}{\sum}|d_i|=\underset{i\in \C}{\sum}|w_i-w_i^b|+\underset{i\notin \C}{\sum}w_i^b\geq\underset{i\in \C}{\sum}(w_i-w_i^b)+(1-\underset{i\in \C}{\sum}w_i^b)=2(1-\underset{i\in \C}{\sum}w_i^b)
\end{equation}
And thus $\underset{i\in \C}{\sum}w_i^b\leq 0.4$ implies the \emph{total absolute deviation constraint}. Therefore to satisfy the \emph{active share constraint}, we randomly deselect assets in $\C$ to reduce the sum of the benchmark weights such that $\underset{i\in \C}{\sum} w_i^b \leq 0.4$. \\

The \emph{tracking error constraint} is attained during each step by re-adjusting the value of $\lambda$ when the constraint is violated. Notice that  $d^T\Omega d$ is one of the objective terms; hence by adjusting $\lambda$ we can make \ref{model:Psim} change the priority: decreasing $\lambda$ gives more importance on minimizing the tracking error, ergo reducing $d^T\Omega d$, while increasing $\lambda$ gives more weight on maximizing the expected revenue by allowing a higher risk, hence increasing $d^T\Omega d$.\\

To satisfy the cardinality constraint, we construct the portfolio by solving problem \ref{model:Psim} to optimality on a set of 70 selected candidate assets  (all the other assets are considered to have weight 0 in the solution). In this way the cardinality constraint is satisfied. Iteratively the candidate set is modified by dropping assets of zero or small weight and replacing them by new candidate assets selected from the given universe of assets.

As we solve problem (\ref{model:Psim}) many times it is important that we use an efficient solver. We use Mosek version 8 \cite{mosek8} as solver under the Yalmip \cite{yalmip} environment release R20180926 in MATLAB 2017b software. This solver was the most efficient for our problem of the solvers we tested.

Now we explain how we initially pick the set of assets and how we update it in each iteration.
For each given review period, we initialize the algorithm by setting $\C$ as the set of assets that were selected in the portfolio of the previous period. In the case when a candidate asset from the previous balance portfolio is no longer in the market, it is replaced in $\C$ by a random asset.

Iteratively, after solving \ref{model:Psim} over a set $\C$ of candidate assets, we remove from $\C$ the assets with the lowest weight and all assets we do not invest in (investment$<10 ^{-5}$) and remove them from $\C$. Then to adjust the cardinality, i.e., the number of assets that we consider in the portfolio to $70$, we reselect assets not in $\C$ to be added to $\C$. We do this based on the marginal effect of investing in those assets. Take an asset $i \notin \C$ and let $\C' = \C \cup \{i\}$. We are interested on knowing wether $i$ will have a positive or 0 weight when solving (\ref{model:Psim}). To check this we fix $w_i = \epsilon$ in (\ref{model:Psim}$^{\epsilon}$). Notice that this new problem is very similar to (\ref{model:Psim}), and we call it (\ref{model:Psim}$^{\epsilon}$). 
 
The objective of  (\ref{model:Psim}$^{\epsilon}$) is the same objective of (\ref{model:Psim}) plus the term $(2d^T\Omega_{\cdot,i} -\lambda \alpha_i)\epsilon + \Omega_{i,i} \epsilon^2$.
 
The marginal direct effect on the objective from including asset $i$ can be split into three main parts. These three parts are: the marginal effect on the portfolio variance, the marginal effect on the mean portfolio return and the marginal effect on the turnover costs. The marginal direct effect on the objective function and the turnover from including asset $i$  is obtained from equation \eqref{eq:M1}.
\begin{equation}\label{eq:M1}
	m_i:=2d^T\Omega_{\cdot,i} -\lambda \alpha_i -2\lambda 10^{-3} \mathbf{I} (w^{Pre}_i\geq 10^{-5}) 
\end{equation} 
In the formulation of marginal costs, $w^{Pre}_i$ denotes the amount invested in asset $i$ in the previous portfolio. Correspondingly, $\mathbf{I} (w^{Pre}_i\geq 10^{-5})$ is an indicator function. This indicator function takes value $1$ if during the previous period a numerically significant amount has been invested in asset $i$; i.e., at least $0.001\%$ of the current budget, and $0$ otherwise. \\
 
The constraints of  (\ref{model:Psim}$^{\epsilon}$) are the same constraints of (\ref{model:Psim}) except for the constant term. 
The indirect effect can be obtained by using the dual (shadow) values. Shadow values are typically described for linear optimization problems, but under strong duality, they can be used for nonlinear models as well. Write the constraints of (\ref{model:Psim}) as $\bar{A}x\leq b$ and let $s$ denote the corresponding dual values. Then the indirect effect of investing in an asset $i$, as in the case of column generation, is given by equation \eqref{eq:M2}.
\begin{equation}\label{eq:M2}
	k_i:=s^T A_{\cdot,i} 
\end{equation}
Then the marginal effect $\delta_i$  of investing in asset $i$ is obtained by using equation \eqref{eq:M3}.
\begin{equation}\label{eq:M3}
	\delta_i=m_i-k_i 
\end{equation}

The assets with the most negative marginal effect are added to the list of candidate assets, such that the cardinality of the new set $\C$ is $70$. Then the convex problem (\ref{model:Psim}) is solved again. We choose allowing to invest in a fixed number of $70$ assets because investing in fewer than $70$ assets leads to a smaller feasible region for (\ref{model:Psim}) and therefore leads to a worse objective than allowing to invest in more assets. Algorithm \ref{alg:main} describes the process of optimizing the portfolio step-by-step. \\

\begin{algorithm}[H]\label{alg:main}
\KwData{\\
\begin{tabularx}{\textwidth}{@{}l<{:}@{\ }X@{}}
$\alpha$&Mean return parameter,\\
$\Omega$&Variance-Covariance matrix,\\
$w_{t-1}$&Previous period portfolio,\\
$w^b$&Benchmark portfolio,\\
$\bar{\lambda}=5$&Risk parameter\\
$\lambda=5$&Adjusted risk parameter\\
$w_{t}^{Pre}$ &Adjusted obtained weights in period $t-1$\\
$Removed$ $=$ $I(w_{t-1}=0)$&Assets in which we do not invest in period $t-1$\\
$nonzeros=\sum_i{\neg Removed_i}$&Number of assets in which we do not invest\\
$w_{t}=Psim:w(\neg Removed)$&Initial portfolio\\
$d_{t}=w_{t}-w^b$&Difference in weights compared to the benchmark \\
$\epsilon$ $=$ $10^{-3}\lambda$ & Turnover penalty\\
\end{tabularx}
}
\KwResult{$w_{t,best}$, $d_{t,best}$}

\While{time $\leq 2.9$ minutes}{
\While{$\sum_{Removed} |w^b(Removed)|< 0.6$}{
	Add arbitrary asset to $Removed$\\
}
Obtain $w_{t}$ and $d_{t}$ from $Psim$ with $w_t(Removed)=0$\\
\eIf{$\sqrt{d^T\Omega d}< 0.05$}{
Set $\lambda=0.9\lambda$\\
}{\eIf{$\sqrt{d^T\Omega d}> 0.1$}{
Set $\lambda=1.1\lambda$\\
}{\If{Psim($w_t$,$\bar{\lambda}$)+$\epsilon$ turnover($w_t$,$w^b$)$<$Psim($w_{t,best}$,$\bar{\lambda}$))+$\epsilon$ $turnover(w_{t,best},w_{t}^{Pre})$}{
Set $w_{t,best}=w_t$\\
}
set $oldRemoved=Removed$\\
set $Removed=oldRemoved \cup \{i:$ $argmin_i w_{t}(i)|i \in \neg Removed$\}\\
set $w_{t}(Removed)=0$\\
set $Removed=(w_t<10^{-5})$\\
set $nonzeros=\sum_i{\neg Removed}$\\
Select $70-nonzeros$ assets based on with best marginal effect at $w_{t}$ on the objective\\
$newRemoved=Removed$ excluding the selected assets\\
\If{$(oldRemoved=newRemoved)$}{
reselect $70-nonzeros$ random non-removed assets\\
$newRemoved=Removed$ excluding the reselected assets\\
}
$Removed=newRemoved$\\

}
}
}
\caption{CG for review period t}
\end{algorithm}

\section{Application}\label{appl}
Next, we apply the solution approach on the data set used and evaluate our solution with relevant performance metrics. 
\subsection{Data}
The problem is defined by Principal Inc, a global investment company. The company also provided the data to test the solution approach. We were provided with 10 years of time series data starting at 2007-01-03, which we refer to as S\&P 500 data, and estimators for the 4 weeks variance-covariance matrix of the return of those assets over the same time interval. The entries updated every 28 days (4 weeks). The elements of each entry for a date are: the identifier which is the unique identifying SEDOL code, the sector of the company which will be used to set the sector active weight in an interval, the beta value which reflects the volatility of an asset compared to the whole market, the alpha score which is the performance estimation shown as the expected return, name of the asset, benchmark weight which is the investment amount in the provided benchmark and the market cap quintile which reflects the size of the asset's company by the quantile in the market. \\

In the results sheet we are given the historical returns for each 4 week period. We assume that the returns on a date are the result of the investment done 4 weeks before. Thus, the return entry at the first date 2007-01-03 is actually the return coming from the investment made on 2006-12-06.
\subsection{Performance Measurements}
The basis of the cost calculations are the portfolio weights $w$ and the four-week returns of the given portfolio $r$. For the calculations of our performance measures we also include the turnover adjusted returns. The turnover is penalized by $0.5\%$ in our method to avoid high costs corresponding to changing portfolio. The turnover is calculated by the following two formulas.
\begin{equation}
w_{i,t}^{Pre}=\dfrac{w_{i,t-1}(1+ r_{i,t-1})}{\sum_i w_{i,t-1}(1+ r_{i,t-1})}
\end{equation}
\begin{equation}
turnover(w_{i,t},w_{i,t}^{Pre})=\sum_i{|w_{i,t}-w_{i,t}^{Pre}|}
\end{equation}
We evaluate our solution by comparing the following results with the benchmark results both including and excluding the turnover costs:
\begin{itemize}
\item \textbf{Cumulative Return} is the total return obtained. It is used to see how well the chosen portfolio scheme has performed at the end of the final period. 
\item \textbf{Annual Return} converts the cumulative period into an average annual return. The advantage of using Annual Returns is that this performance measure is independent of the measured time interval. 
\item \textbf{Annualized Excess Return} is defined as the difference between the annual return of the proposed portfolio and the annual return of the benchmark.
\item \textbf{Tracking Error} is the standard deviation of the difference between the returns of the chosen portfolio and the benchmark. We use \textbf{Annualized Tracking Error} by considering the number of re-balances in a year. The Tracking error can be used to analyze how closely the portfolio follows the benchmark.
\item \textbf{Sharpe Ratio} is a measure to show the return of the portfolio compared to the risk it carries \cite{sharpe1994sharpe}. It is computed by subtracting the best risk-free option from the final return and dividing it by the standard deviation of the observed returns. Since a risk-free option is not investing at all, we take this value as 1. This ratio helps us to understand how choosing riskier assets affect the extra return we have compared to the risk-free option. This ratio is helpful to see the trade-off between the return and the risk.
\item \textbf{Information Ratio} is a measure to show how good the portfolio returns compared to the benchmark given, with a tracking error normalization. According to \citeA{grinold2000active}: `` The information ratio measures achievement ex post (looking backward) and connotes opportunity ex ante (looking forward)." It is found by dividing the difference between the returns of portfolio and benchmark by the tracking error. Information Ratio is similar to Sharpe Ratio, however IR gives the risk and return by taking benchmark as the base case. Higher IR is given by higher difference between portfolio and benchmark, and a lower tracking error. So higher IR can be used to see ``how closely the benchmark is followed, with how much better return.''
\end{itemize}

\subsection{Results and Analysis}
We used a run-time limit of $170$ seconds for each review period for the algorithm. To check the performance of our method we compare our portfolio's performance on the Principal dataset to the benchmark. In this comparison, we do not consider turnover costs for the returns of the benchmark portfolio. We therefore split our result analysis into two parts. We first compare the results of our method excluding turnover costs. In the second part, we do include the turnover costs in our portfolio. 

It can be seen from our performance statistics in Table \ref{tab1} that the result of our method excluding turnover costs outperforms the benchmark. In each of the performance statistics, our model's results are better than the performance statistics of the benchmark.
\begin{table}[H]
\caption{Portfolio Performance Statistics Excluding Turnover costs.}
\centering
\begin{tabular}{|l|c|c|}
\hline
\textbf{2007-01-01 to 2016-12-31}&
\textbf{Portfolio}&
\textbf{Benchmark} \\
\hline
Cumulative Return&
{    275.65\%} & 
{239.10\%}\\
\hline
Annualized Return&
{14.15\%}& 
{12.99\%} \\
\hline
Annualized Excess Return&
{1.16\%}&
-- \\
\hline
Sharpe Ratio&
55.75&
40.81 \\
\hline
Information Ratio&
24.25&
-- \\
\hline
\end{tabular}
\label{tab1}
\end{table}
Our model handles turnover cost indirectly but the obtained portfolios still have a large turnover. This can also be seen in Figure \ref{fig3}, which plots the turnover costs of our portfolios over the review periods. These turnover costs are calculated as a $0.5\%$ cost per unit turnover. Figure \ref{fig3} indicates that our method does not focus enough on having a small turnover. 
\begin{figure}[H]
\begin{center}
\includegraphics[scale=0.9]{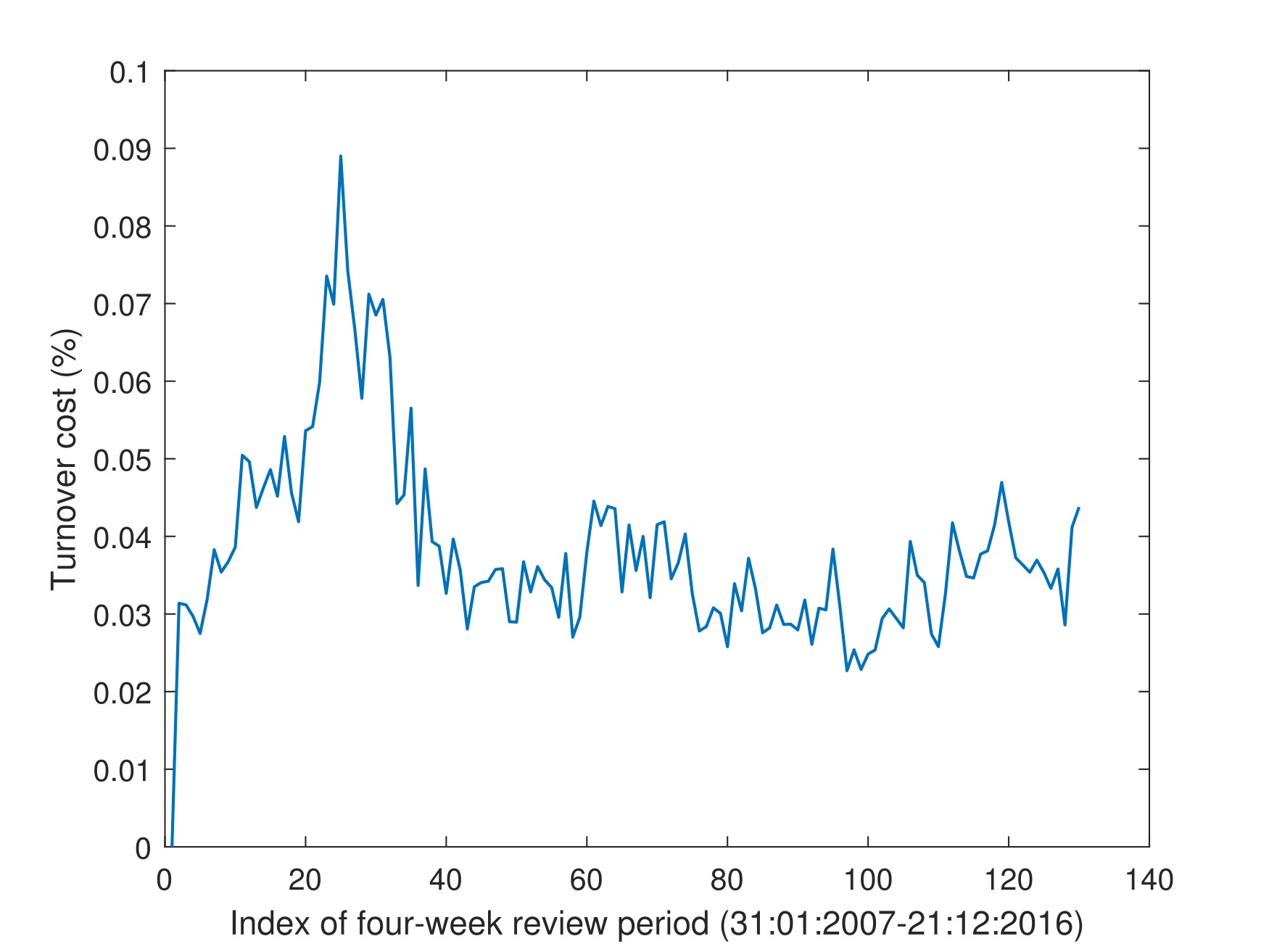}\\
\caption{Turnover costs: based on $0.5\%$ turnover costs per unit turnover}
\label{fig3}
\end{center}
\end{figure}
The relatively high turnover costs of our portfolios lead to our method performing slightly worse than the benchmark. These performance statistics can be seen in Table \ref{tab2}. These differences in performance suggest that our method adjusts the portfolio too much each period. That is why, for future development of this model, it is important to focus more on the minimization the turnover costs. In the section \textit{Discussion}, we propose a few adjustments to the method to achieve lower turnover.
\begin{table}[H]
\def\arraystretch{1.4}
\begin{center}
\caption{Portfolio Performance Statistics Including Turnover costs}
\begin{tabular}{|l|c|c|}
\hline
\textbf{2007-01-01 to 2016-12-31}&
\textbf{Portfolio}&
\textbf{Benchmark} \\
\hline
Cumulative Return&
{    207.18\%} & 
{239.10\%}\\
\hline
Annualized Return&
{11.88\%}& 
{12.99\%} \\
\hline
Annualized Excess Return&
{--1.11\%}&
-- \\
\hline
Annualized Tracking Error&
{5.54\%}&
-- \\
\hline
Sharpe Ratio&
41.95&
40.81 \\
\hline
Information Ratio&
--20.75&
-- \\
\hline
\end{tabular}
\label{tab2}
\end{center}
\end{table}
The difference in returns is due to the $0.5\%$ turnover cost per unit turnover. Figure 3 illustrates the turnover costs per four week review date beginning in $31^{th}$ January $2007$ to $21^{th}$ December $2016$. Figure \ref{fig1} shows that the resulting returns including turnover costs are on average not as high as the benchmark. This is in part compensated by a smaller risk, which is shown as a smaller spread of the box plot. Therefore our resulting portfolio can be seen as more robust. 
\begin{figure}[H]
\begin{center}
\includegraphics[scale=0.9]{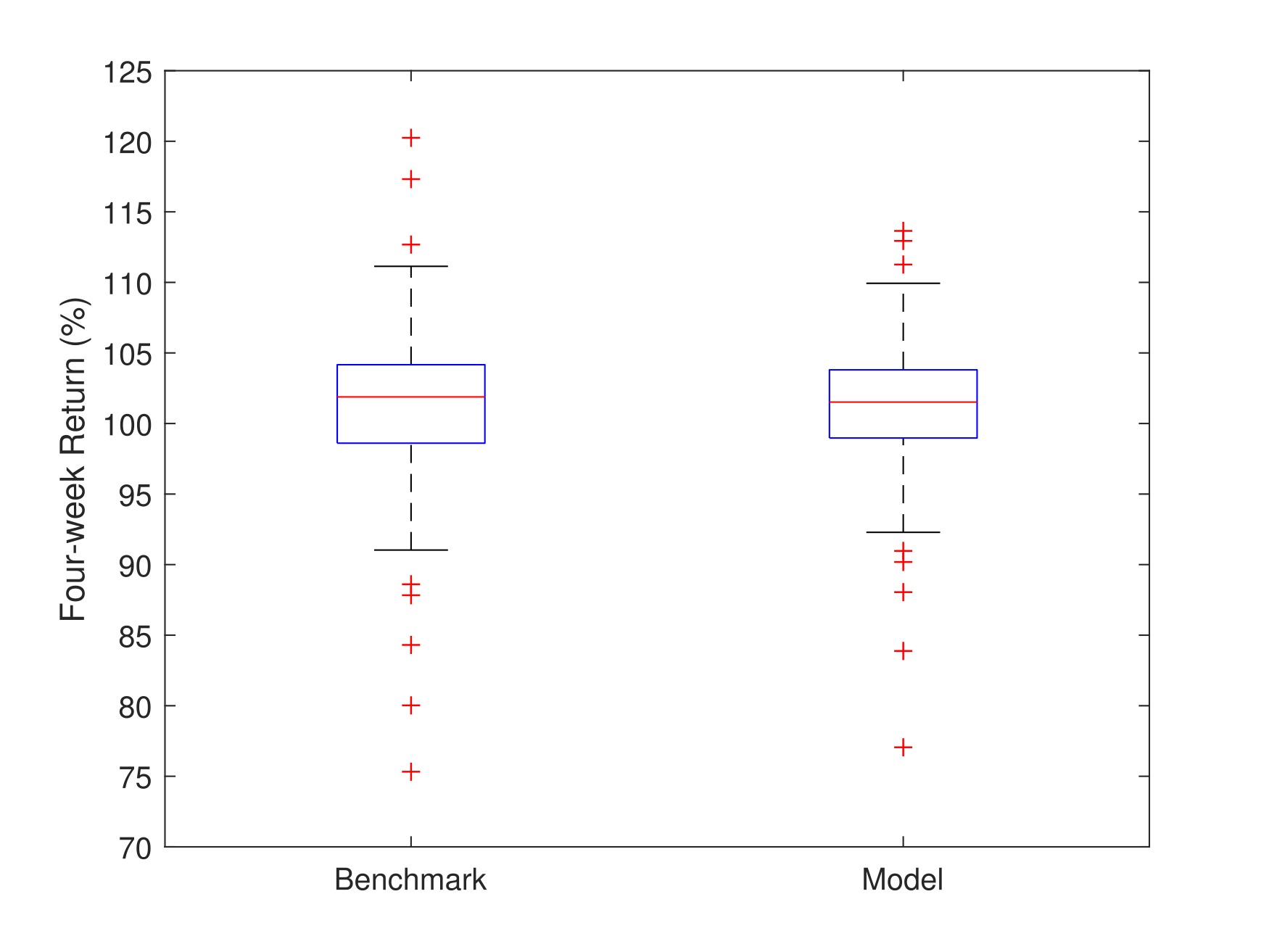}\\
\caption{Boxplot of the turnover adjusted 4-week returns}
\label{fig1}
\end{center}
\end{figure}

Next, we show the performance of our method over each review period. These results are shown in Figure \ref{fig2}, which shows that the model follows the benchmark quite closely and our portfolios even lead to smaller peaks. This can be interpreted as our method leads to more robust portfolios than the benchmark. 
\begin{figure}[H]
\begin{center}
\includegraphics[scale=0.9]{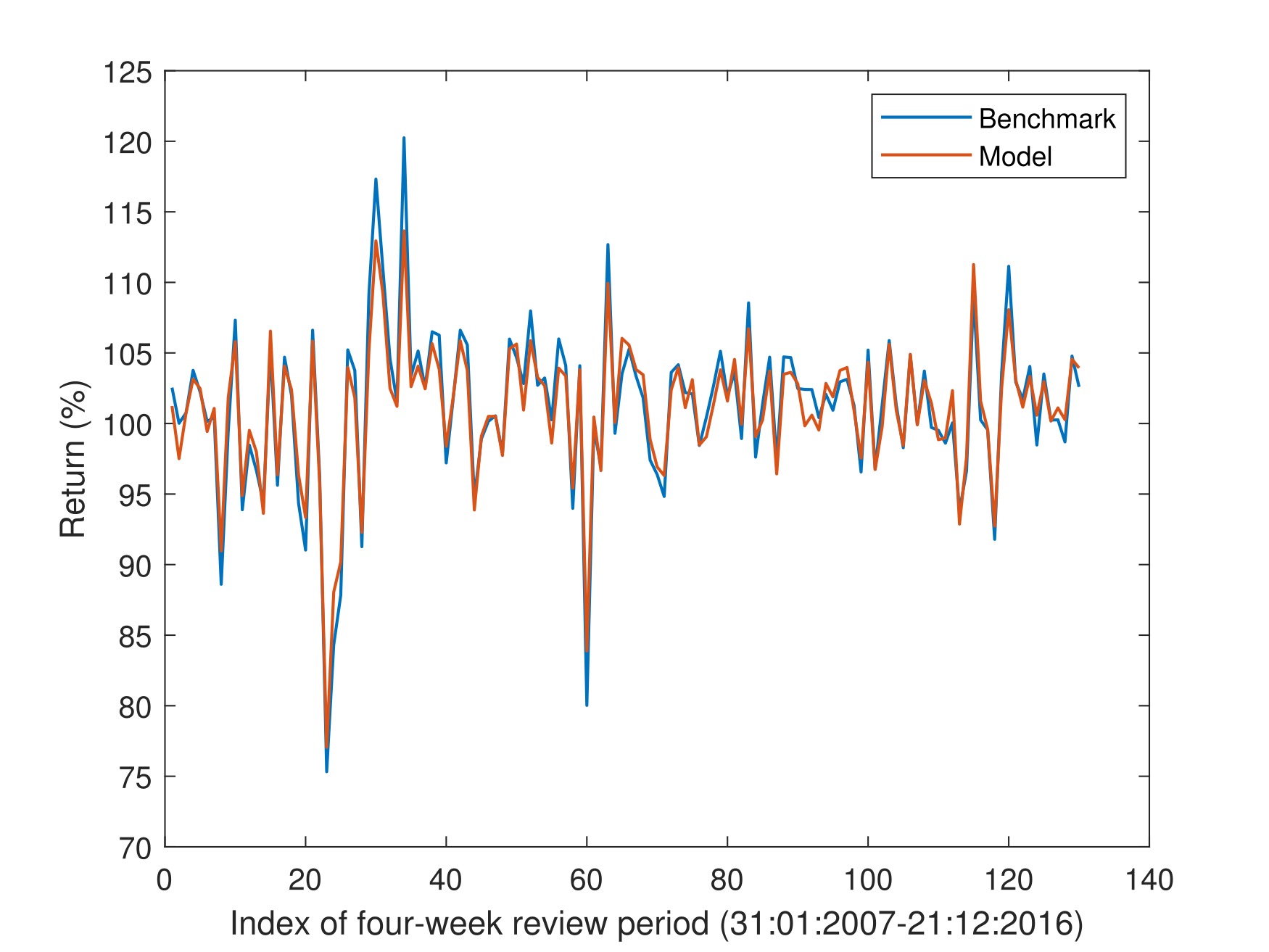}\\
\caption{Turnover adjusted 4-week returns per review period}
\label{fig2}
\end{center}
\end{figure}

\section{Discussion}\label{disc}
Our method shows promise for solving portfolio optimization problems with cardinality constraints. Table \ref{tab1} shows that the performance without the turnover costs is better than the benchmark, in terms of both risk and expected return. But the method could be improved in terms of turnover. Figure 3 and Table \ref{tab2} both show that the method has difficulties with the limitations of the turnover cost. Therefore, for future research we recommend to constrain the allowed turnover. Although the turnover adjusted performance in terms of cumulative return is slightly worse than the benchmark, this is partially due to the absence of the turnover costs for the benchmark. Overall, this method incorporates a comprehensive model that performs close to the benchmark keeping many practical issues in consideration.  
The method should work similarly on quadratic optimization problems (in particular in linear problems) with cardinality constraints.

Since in our model we do not allow infeasible solutions to exist, it is necessary to make our model more flexible or to allow (small) violations of the constraints, if no solution has been found. Another improvement that could be made is to consider multiple period portfolio selection to reduce the significant turnover costs. The turnover costs can also be reduced by considering the changes in the performance parameters $\alpha$, and $\Omega$, of the assets over time, which can be used to make more consistent portfolios or find trends in the performance of the assets.

One possible variation of the method is a column generation algorithm based on a random selection of the new assets; this would make the solution approach more diversified since in each iteration there would be more potential assets. Also, this may help to escape possible local optima. The idea is to select the assets which leave and enter the set of candidate assets  randomly. They could be chosen uniformly or with probabilities proportional to the marginal effects or Reduced Cost, where the Reduced Costs are determined by the shadow values. \\

In our algorithm, for each time period we have the initial portfolio referring back to the previous portfolio and the algorithm changes it by improvements. However, since we initiate the previous portfolio, the portfolio will end up close to the previous portfolio. With this way, we made sure the algorithm is efficient and can be terminated quickly. This allows the model to produce well-based results on even a much larger scale dataset (10-fold or 20-fold). On the other hand, one could allow to look at every possible asset to initialize the portfolio for the algorithm in each time period. This will increase the computational burden, but will allow to diversify the selections more.\\

\clearpage
\bibliographystyle{apacite}
\bibliography{interactapasample}
\end{document}